\documentclass[psamsfonts]{amsart}

\usepackage{amssymb,amsfonts}
\usepackage{amsmath}
\usepackage[all,arc]{xy}
\usepackage{enumerate}
\usepackage{mathrsfs}
\usepackage{mathtools}
\usepackage{cite}
\usepackage[unicode]{hyperref}
\usepackage{bookmark}

\newtheorem{thm}{Theorem}[section]
\newtheorem{theorem}[thm]{Theorem}
\newtheorem{cor}[thm]{Corollary}
\newtheorem{prop}[thm]{Proposition}
\newtheorem{lemma}[thm]{Lemma}

\newtheorem{quest}[thm]{Question}
\newtheorem{claim}[thm]{Claim}

\theoremstyle{definition}
\newtheorem{defn}[thm]{Definition}

\newtheorem{exmp}[thm]{Example}

\newtheorem{rem}[thm]{Remark}

\newcommand{\la}{\lambda}
\newcommand{\ka}{\kappa}
\newcommand{\w}{\omega}
\newcommand{\sm}{\setminus}
\newcommand{\mc}{\mathcal}

\DeclareMathOperator{\Lev}{Lev}

\DeclareMathOperator{\ot}{ot}

\DeclareMathOperator{\fd}{fd}
\DeclareMathOperator{\Pred}{Pred}
\DeclareMathOperator{\hgt}{ht}
\DeclareMathOperator{\width}{width}

\DeclareTextCommand{\textaleph}{L8U}{ℵ}

\newcommand{\uhr}{\upharpoonright}

\makeatletter
\let\c@equation\c@thm
\makeatother
\numberwithin{equation}{section}

\title{A Note on Narrow Systems}
\author{William Adkisson}
\begin{document}

\begin{abstract}
 We examine the relationship between trees and narrow systems. In particular, we prove that if a tree of height $\ka$ admits a narrow subsystem, it must contain a well-pruned subtree of height $\ka$. We extend these results to two-cardinal trees.
\end{abstract}

\maketitle
\section{Introduction}

The tree property, an uncountable generalization of K\"onig's Lemma, states that all trees of height $\ka$ with levels of size $<\ka$ have an unbounded branch. The tree property is closely tied to large cardinals: an inaccessible cardinal $\ka$ is weakly compact if and only if it has the tree property. Moreover, while the tree property can hold at smaller cardinals, the consistency strength of the tree property at non-inaccessibles is at least a weakly compact.

A generalization of the tree property, the strong tree property, has a similar relationship with strongly compact cardinals: Jech \cite{jech:combinatorialprobs} proved that an inaccessible cardinal $\ka$ is strongly compact if and only if it has the strong tree property. While the strong tree property can hold at small cardinals, the only known methods for forcing it require (at least) a strongly compact cardinal.

A famous theorem of Solovay \cite{solovay:sch} states that the Singular Cardinal Hypothesis (SCH) always holds above a strongly compact cardinal. Since the strong tree property is so closely tied to strongly compacts, a natural question is the following:

\begin{quest}\label{q:doesTPimplySCH}
	Is it consistent for SCH to fail above a cardinal with the strong tree property?
\end{quest}

There have been many results proved over the years, refining exactly how much large cardinal strength is required to obtain SCH. Viale \cite{Viale:PFA} proved that the Proper Forcing Axiom implies SCH above $\aleph_2$, and PFA implies that the strong tree property holds at $\aleph_2$. Work of Krueger \cite{krueger:guessing} and Viale\cite{viale:guessing} implies that if ISP, a powerful strengthing of the strong tree property linked to supercompact cardinals, holds at a cardinal $\ka$, SCH must hold above $\ka$.

Recently, Lambie-Hanson\cite{lambie-hanson:narrowsystems} proved that if a certain two-cardinal narrow system property holds at $\ka$, then SCH (and in fact Shelah's Strong Hypothesis, a technical strengthening of SCH) must hold above $\ka$. This narrow system property holds in all known models of the strong tree property, and in some cases is used as a crucial ingredient in the proof that the strong tree property holds. Moreover, it is usually easier to obtain, and holds in some models where the strong tree property fails. This is very strong evidence towards a positive answer to Question \ref{q:doesTPimplySCH}. It is open if the strong tree property implies this narrow system property; if it does, then Lambie-Hanson's result would answer Question \ref{q:doesTPimplySCH} immediately.

Lambie-Hanson's argument involves constructing a specific concrete systems that, if they contain branches, mandate SCH to hold. We will show that these systems can be used to generate certain $(\ka,\la)$-trees in a canonical way. A branch through the generated tree does not necessarily induce a branch through the original system, but it does imply some amount of structure, and this is enough to finish Lambie-Hanson's argument.

Although this allows us to move from systems to the trees that they generate, the generated $(\ka,\la)$-trees are not necessarily thin, so the strong tree property may not imply that they have a branch. This motivates the following question:

\begin{quest}
	Assuming the strong tree property, do all $(\ka,\la)$-trees generated from narrow concrete $(\ka,\la)$-systems have a cofinal branch?
\end{quest}

A positive answer to this question would imply a positive answer to Question \ref{q:doesTPimplySCH}. There are a number of ways to modify the question that would still give interesting information. The system used in Lambie-Hanson's argument has a very specific format, closely related to subadditive colorings; we can restrict our question to only these systems. We could also increase our assumption to strengthenings of the strong tree property like the super tree property, since Question \ref{q:doesTPimplySCH} is also open for that principle.

We are also interested in a much more general version of this question:
\begin{quest}
	Which trees admit narrow subsystems?
\end{quest}
Understanding the answer to this question has implications beyond studying SCH. All known methods for obtaining the tree property and the strong tree property at the successor of a singular cardinal use a two-step process, first finding a narrow subsystem of a tree, and then obtaining a branch through that system. The process of finding the narrow system leads to most of the difficulties of the proof and requires many large cardinals; a more streamlined way of detecting narrow systems could resolve some of the major open questions related to the (strong) tree property at successors of singulars.

In this paper, we begin to examine the relationship between trees (both standard trees and two-cardinal trees) and narrow systems. In particular, we show that if a tree $T$ admits a narrow subsystem, it must have a well-pruned subtree, and generalize this result to $(\ka,\la)$-trees.

\section{Basic Definitions: Trees and Systems}

\subsection{One-cardinal objects}
\begin{defn}
	A partial order $(T, \leq_T)$ is a \emph{tree} if it has a minimal element (known as the root), and for all $y\in T$, the set $\Pred(y) = \{x\in T \mid x\leq_T y\}$ is well-ordered.
	
	The \emph{height} of an element $x\in T$ is $\hgt(x) = \ot(\Pred(x))$. The \emph{$\alpha$-th level} of $T$ is $T_\alpha = \{x\in T \mid \hgt(x) = \alpha\}$. The \emph{height} of a tree $T$ is $\sup_{x\in T}(\hgt(x)+1)$.
	
	We say that a tree of height $\kappa$ is \emph{thin} if all levels have size $<\ka$.
	
	The \emph{cone} above a node $x$ is $C(x) = \{y\in T \mid x\leq_T y\}$. If $x$ appears in multiple trees, we will sometimes write this as $C_T(x)$ to clarify which tree we are drawing elements from.
	
	We say that a cone $C$ above a node $x \in T_\alpha$ is \emph{unbounded} if $C(x)\cap T_\beta$ is nonempty for all $\beta >\alpha$.
\end{defn}

\begin{defn}
	Let $I \subseteq \ka$ be unbounded. A sequence $S = \langle S_\alpha \mid \alpha \in I \rangle$, equipped with a partial order $\leq_S$ on $\bigcup_{\alpha\in I} S_\alpha$, is a \emph{system of one relation} with height $\ka$ and index set $I$ if the following hold:
	\begin{enumerate}
		\item Let $\alpha,\beta \in I$. Let $x \in S_\alpha$ and $y \in S_\beta$. If $x \leq_S y$, then either $\alpha <\beta$ or $\alpha = \beta$ and $x=y$.
		\item Let $\alpha_0$ and $\alpha_1$ both be below $\beta$. Then $x_0 \in S_{\alpha_0}$, $x_1 \in S_{\alpha_1}$ and $y\in S_\beta$, $x_0$ and $x_1$ are comparable.
		\item For all $\alpha < \beta$ both in $I$, there exists $x \in S_\alpha$ and $y\in S_\beta$ such that $x \leq_S y$.
	\end{enumerate}

	We will abuse notation and use $S$ to refer to both the system itself and the set $\bigcup_{\alpha\in I} S_\alpha$.
	
	The \emph{width} of a system is $\width(S) = \sup_{\alpha\in I}(|S_\alpha|)$. A system is \emph{thin} if $|S_\alpha|<\ka$ for all $\alpha \in I$, and \emph{narrow} if $\width(S)^+ < \ka$.
	
	A \emph{cofinal branch} through $S$ is a set $b = \{x_\alpha \mid \alpha \in I\}$ such that each $x_\alpha \in S_\alpha$, and for all $\alpha < \beta$ both in $I$, $x_\alpha \leq_S x_\beta$.
	
	The \emph{narrow system property} holds at $\ka$ if every narrow system of height $\ka$ with one relation has a cofinal branch.
\end{defn}

	\begin{rem}
	For both systems and trees, by enumerating the levels $T_\alpha$ and $S_\alpha$, we can assume that the $\zeta$-th element of the $\alpha$-th level is of the form $(\alpha, \zeta)$.
\end{rem}

As the name would suggest, systems with one relation are a special case of a more general notion, systems of relations, due to Magidor and Shelah \cite{MS_TPSuccSing}. When examining the narrow system property, it suffices to restrict to a single relation; see \cite[Remark 28]{lambie-hanson:squares} for details.

A system can be thought of as a tree after some of its have been forgotten. Given any two levels, there needs to be \textit{some} node on the upper level above some element of the lower level, but the system may not know about any nodes in between the two on intermediate levels. In a tree, on the other hand, any pair of levels must have \emph{every} node on the upper level must be above some node on the lower level.

\begin{defn}
	Let $(T, \leq_T)$ be a tree of height $\ka$. We say that a system $S$ is a subsystem of $T$ if each $S_\alpha \subseteq T_\alpha$, and $\leq_S = \leq_T\uhr \bigcup_{\alpha < \ka} S_\alpha$.
\end{defn}

We will be particularly interested in the following kind of system.

\begin{defn}
	Let $\ka$ be regular. A \emph{concrete $\ka$-system} is a sequence $S = \langle S_\alpha \mid \alpha \in I\rangle$ such that:
	\begin{enumerate}
		\item $I$ is a cofinal subset of $\ka$.
		\item For all $\alpha \in I$, $S_\alpha$ is a nonempty collection of subsets of $\alpha$.
		\item For all $\alpha \leq \beta$ both in $I$, there is $y\in S_\beta$ such that $y\cap \alpha \in S_\alpha$.
	\end{enumerate}
\end{defn}
A concrete $\ka$-system can be viewed as a system of one relation with height $\ka$, whose elements are pairs $(\alpha, x)$ for $x\in S_\alpha$. The relation $\leq_S$ is defined in the following way: if $(\alpha,x)\in S_\alpha$ and $(\beta,y)\in S_\beta$, $(\alpha,x)\leq_S (\beta,y)$ iff $\alpha \leq \beta$ and $y\cap \alpha = x$.

Because concrete systems are so, well, concrete, they allow us to generate a tree in a canonical way.

\begin{defn}
	Let $S$ be a concrete $\ka$-system. The \emph{tree generated from $S$} is the tree $T$ whose levels are defined in the following way. The $\alpha$-th level of $T$ consists of $\{(\alpha,x\cap \alpha) \mid x\in S_\beta \text{ for some } \beta \geq \alpha\}$. If $(\alpha, x) \in T_\alpha$ and $(\beta, y) \in T_\beta$, then $(\alpha,x)\leq_T (\beta,y)$ if and only if $\beta \geq \alpha$ and $x = y\cap \alpha$.
\end{defn}
This tree is essentially just the downwards closure of the system. If a system is a tree after we forget some of its nodes, the generated tree restores the forgotten nodes. For concrete systems, unlike more general systems, this can be done canonically.

This lets us ask a slightly more specific question:
\begin{quest}
	Which trees (up to isomorphism) can be generated by narrow concrete $\ka$-systems?
\end{quest}

This question is made nontrivial by the fact that even if a concrete system is narrow, its generated tree may not be thin.

\begin{exmp}
	Partition $\ka$ into two unbounded sets, $A_0$ and $A_1$. Let $I = \{\alpha+1 \mid \alpha \in A_1\}$. Further partition $A_1$ into $\ka$-many unbounded sets $B_\alpha$. Let $x^0_{\alpha+1} = A_0\cap (\alpha+1)$. Let $x^1_{\alpha+1} = B_\beta\cap (\alpha+1)$, where $\beta$ is the unique ordinal such that $\alpha \in B_\beta$.
	
	We define a concrete $\ka$-system $S = \langle \{x^0_\alpha, x^1_\alpha\} \mid \alpha \in I\rangle$. This is a concrete system with width $2$, since the coherence requirement is satisfied by the $x_\alpha^0$'s.
	
	Despite $S$ having such small width, the tree it generates is not thin. The tree generated from $S$ consists of one branch of height $\ka$ whose elements are initial segments of $A_0$, along with a branch of height $\alpha$ extending from the root for each $\alpha$ in the index set $I$ whose elements are initial segments of the $B_{\beta}$'s. Each level has size $\ka$, even though the original system only had width $2$.
\end{exmp}

\subsection{Two-cardinal objects}
Now we turn our attention to more general objects, whose levels are indexed by elements of $\mc{P}_\ka(\la)$ rather than by $\ka$.

\begin{defn}
	Let $\ka$ be a regular uncountable cardinal, with $\ka \leq \la$. A \emph{$(\ka, \la)$-tree} is a sequence $T = \langle T_z \mid z\in \mc{P}_\ka(\la)\rangle$ with the following properties:
	\begin{itemize}
		\item For all $z \in \mc{P}_\ka(\la)$, $T_z$ is a nonempty collection of subsets of $z$.
		\item If $z \subseteq y$ are both in $\mc{P}_\ka(\la)$, then for all $x \in T_y$, $x\cap z \in T_z$.
	\end{itemize}
	The entries $T_z$ are called the \emph{levels} of $T$. There is a natural tree-like order $\leq_T$ given by this structure: if $z \subseteq z'$ and $x\in T_z$, $y\in T_{z'}$, then $x \leq_T y$ iff $y\cap z = x$.
	
	A $(\ka,\la)$-tree is \emph{thin} if $|T_z|<\ka$ for all $z\in \mc{P}_\ka(\la)$.
	
	A \emph{cofinal branch} through $T$ is a subset $b\subseteq \la$ such that $b\cap z \in T_z$ for all $z\in \mc{P}_\ka(\la)$.
	
	The \emph{strong tree property}\footnote{Many papers on the strong tree property, including several written by the author, define it in terms of different objects, called thin $\mc{P}_\ka(\la)$ lists, rather than thin $(\ka,\la)$-trees. The two definitions are equivalent.} holds at $\ka$ if for all $\la \geq \ka$ and all thin $(\ka,\la)$-trees $T$, there is a cofinal branch through $T$.
\end{defn}

\begin{defn}\cite[Definition 2.1]{lambie-hanson:narrowsystems}
	Let $\ka$ be a regular uncountable cardinal, and let $\la \geq \ka$. A concrete $\mc{P}(\ka,\la)$-system is a sequence $S = \langle S_z\mid z\in I\rangle$ such that:
	\begin{enumerate}
		\item $I$ is a $\subseteq$-cofinal subset of $\mc{P}_\ka(\la)$.
		\item For all $z\in I$, $S_z$ is a nonempty collection of subsets of $\mc{P}_\ka(\la)$.
		\item For all $z \subseteq z'$ both in $I$, there is $y \in S_{z'}$ such that $y\cap z \in S_z$.
	\end{enumerate}
	
	The \emph{width} of $S$ is $\width(S) = \sup_{x\in I} |S_x|$. $S$ is \emph{narrow} if $\width(S)^+<\ka$. A \emph{cofinal branch} through $S$ is a set $b\subseteq \la$ such that $\{b\cap z \in S_z\}$ is $\subseteq$-unbounded in $\mc{P}_\ka(\la)$.
\end{defn}

As before, we can define the $(\ka,\la)$-tree generated from a concrete $\mc{P}_\ka(\la)$-system. Once again, this is essentially just taking the downwards closure of the system.

\begin{defn}
	Suppose $S = \langle S_z \mid z \in I\rangle$ is a concrete $\mc{P}_\ka(\la)$-system. We define the \emph{$(\ka,\la)$-tree $T$ generated from $S$} as follows. For all $x\in \mc{P}_\ka(\la)$, we set $T_x = \{y\cap x \mid \exists z\in I, z\supseteq x \text{ and } y\in S_z\}$.
\end{defn}

Note that if a concrete $\mc{P}_\ka(\la)$-system $S$ has a branch $b$, the downwards closure of $b$ will give a branch through the generated tree. The opposite may not be true: a branch through the generated tree does not necessarily give rise to a branch in the original system, since the nodes in the branch of a generated tree may not appear in the system itself.

Nonetheless, the existence of a branch through the generated tree has implications for the structure of the underlying system. As an example, we record the following slight strengthening of \cite[Theorem A]{lambie-hanson:narrowsystems}; it is implicit in the proof of that theorem.
\begin{theorem}\label{thm:LH}
	Let $\ka$ be regular. Suppose that for all $\la \geq \ka$ and all concrete narrow $\mc{P}_\ka(\la)$-systems $S$, the $(\ka,\la)$-tree generated by $S$ has a cofinal branch. Then SCH holds above $\ka$.
\end{theorem}

 With this in mind, we can pose the general question:
\begin{quest}
	Does the strong tree property at $\ka$ imply the hypotheses of Theorem \ref{thm:LH}?
\end{quest}

Since narrow $\mc{P}_\ka(\la)$-systems do not necessarily generate thin $(\ka,\la)$-trees, this is akin to asking:

\begin{quest}
	Do all $(\ka,\la)$-trees generated from narrow systems contain a thin subtree?
\end{quest}
Note that, as in \cite{lambie-hanson:narrowsystems}, we can weaken this hypothesis slightly. We only require branches through narrow systems constructed from subadditive colorings, and can restrict to systems with width $\w$.

\section{One-Cardinal Trees and Narrow Subsystems}\label{s:1card}
In this section, we begin to examine which trees can contain narrow subsystems. As a first observation, note that not every tree admits a narrow subsystem. For instance, consider the following tree. 

\begin{defn}
	The \emph{fan of height $\ka$} is the tree consisting of $\ka$-many separate branches extending from the root node, with the $\alpha$-th branch having height $\alpha$.
\end{defn}

Clearly, a fan of height $\ka$ is a tree with height $\ka$ and levels of size $\ka$. If $\ka$ is regular, it has no thin subtree of height $\ka$. In fact, the same is true of subsystems:

\begin{prop}\label{prop:1-fan}
	If $\ka$ is regular, a fan of height $\ka$ has no narrow (or even thin) subsystem.
\end{prop}
\begin{proof}
	Let $T$ be the fan of height $\ka$. Suppose that $S$ is a thin subsystem of $T$, with index set $I$. Consider $S_\alpha$ for some $\alpha \in I$ with $\alpha>0$. For all $\beta \in I$ with $\alpha < \beta$, there must be some $x_\beta \in S_\alpha$ and $y_\beta \in S_\beta$ such that $x_\beta \leq_S y_\beta$.
	
	Since $S$ is thin, $\{x_\beta \mid \beta \in I\sm\alpha\}$ has size $<\ka$. Then there is some $x\in S_\alpha$ such that $x = x_\beta$ for unboundedly many $\beta$. This means that $x$ has unboundedly many nodes above it in $S$ and thus in $T$. But no node in $T$ other than the root satisfies this property, a contradiction.
\end{proof}

The key property we used in this proof was the existence of some $\alpha$ such that every node in $S_\alpha$ had a bounded cone in $T$.
If we want to construct a similar tree without this property, we can start with a fan of height $\ka$ and, at the end of each branch, attach another fan of height $\ka$. Then every node will have an unbounded cone; but these cones (except at the root) will be either bounded or a $1$-fan. We call such a tree a $2$-fan. 

This construction can be iterated transfinitely. At successor stages $\alpha+1$, we attach a copy of the $1$-fan to the end of every branch of the $\alpha$-fan. At limit stages $\gamma$, we will start with a root node with $\gamma$ successors; at the $\alpha$-th successor, we attach an $\alpha$-fan.

For any $\alpha$, the resulting $\alpha$-fan has the property that the cone above any node (that is not the root) will either be bounded or be an $\beta$-fan for some $\beta < \alpha$.

Motivated by this construction, we give the following general definition.

\begin{defn}
	Fix $\ka$ regular, and let $T$ be a tree of height $\ka$. We define the property \emph{$\alpha$-fanlike} for cones in $T$ inductively as follows.
	\begin{itemize}
	\item A cone is $0$-fanlike if it is bounded.
	
	\item A cone $C$ above a node $x$ on level $\beta$ is $(\alpha+1)$-fanlike if there is $\delta$ with $\beta \leq \delta < \ka$ such that for each node $y \in T_\delta\cap C$, the cone above $y$ is $\alpha$-fanlike.
	
	\item If $\gamma$ is a limit ordinal, a cone $C$ above a node $x$ on level $\beta$ is $\gamma$-fanlike if there is $\delta$ with $\beta \leq \delta < \ka$ such that for all $y \in T_\delta\cap C$, the cone above $y$ is $\alpha$-fanlike for some $\alpha < \gamma$.
	\end{itemize}
\end{defn}

\begin{defn}
	We say that a tree $T$ of height $\ka$ with root $r$ is \emph{$\alpha$-fanlike} if $C_T(r)$ is $\alpha$-fanlike. That is, if $T$ is $\alpha$-fanlike when we view it as the cone above the root. $T$ is \emph{fanlike} if it is $\alpha$-fanlike for any $\alpha$.
\end{defn}

Note that this condition weakens as $\alpha$ increases: if $\alpha < \beta$, and $T$ is $\alpha$-fanlike, then $T$ is $\beta$-fanlike as well.

\begin{defn}
	Let $\ka$ regular, and let $T$ be a tree of height $\ka$. We say that $T$ is \emph{fanlike} if $T$ is $\alpha$-fanlike for some ordinal $\alpha$. The \emph{fan degree} of a fanlike tree $T$, written $\fd(T)$, is the least $\alpha$ such that $T$ is $\alpha$-fanlike.
\end{defn}

Fan degrees are closely linked with the notion of a well-pruned tree.

\begin{defn}
	A tree $T$ is well-pruned if every node extends to an element on every higher level.
\end{defn}

To establish this connection, we describe a ``pruning'' process. This process will iteratively remove nodes with bounded cones, until either the entire tree is destroyed or we arrive at a well-pruned subtree.

\begin{defn}
	Let $T$ be a tree of height $\ka$. We define the $\alpha$-th prune of $T$, written $T^\alpha$, inductively. Let $T^0 = T$. At successor stages $\alpha$, we define $B_\alpha = \{x \mid C_{T^\alpha}(x) \text{ is bounded}\}$, and let $T^{\alpha+1} = T\sm \bigcup_{x\in B_0}C_T(x)$. That is, we ``prune'' $T_\alpha$ by removing all nodes that have bounded cones. At limit stages, we define $T^\gamma = \bigcap_{\alpha<\gamma} T^\alpha$.
\end{defn}

\begin{rem}
	Note that if $T$ is well-pruned, then this process does nothing: $B_\alpha$ will be empty for all $\alpha$.
\end{rem}

If $T$ is fanlike, its fan degree measures how many iterations of the pruning process the tree will survive.

\begin{lemma}\label{lem:pruningsurvivors}
	Let $x \in T$. Then $x\in T^{\alpha}$ if and only if $C_T(x)$ is not $\beta$-fanlike (in $T$) for any $\beta < \alpha$.
\end{lemma}
\begin{proof}
	We prove this by induction on $\alpha$. 
	Our base case is $\alpha = 1.$ By definition, $x \in T^1$ if and only if $C_{T^0}(x) = C_T(x)$ is not bounded, which by definition holds if and only if $C_T(x)$ is not $0$-fanlike.
	
	Next, the successor stage. It suffices to show that $x \in T^{\alpha+1}$ if and only if $C_T(x)$ is not $\alpha$-fanlike. By definition, $x \in T^{\alpha+1}$ if and only if $C_{T^\alpha}(x)$ is unbounded. Applying our inductive hypothesis, for all $y \in T^{\alpha}$, $C_T(y)$ is not $\beta$-fanlike for any $\beta < \alpha$. We conclude that $C_{T^\alpha}(x)$ is unbounded if and only if all levels in $C_T(x)$ above the level of $x$ contain some element $y$ such that $C_T(y)$ is not $\beta$-fanlike for any $\beta < \alpha$. This occurs if and only if $C_T(x)$ is not $\alpha$-fanlike.
	
	Finally, let $\gamma$ be a limit ordinal. Then $x\in T^\gamma$ if and only if $x \in T^\alpha$ for all $\alpha < \gamma$, so by our inductive hypothesis $x\in T^\gamma$ if and only if $C_T(x)$ is not $\beta$-fanlike for any $\beta < \gamma$.
\end{proof}

\begin{cor}\label{cor:flpruning}
	$T$ is $\alpha$-fanlike if and only if $T^{\alpha+1}$ is empty.
\end{cor}
\begin{proof}
	Suppose $T$ is $\alpha$-fanlike. Let $x$ be the root of $T$. Then $C_T(x) = T$ is $\alpha$-fanlike, so by Lemma \ref{lem:pruningsurvivors}, $x$ is not in $T^{\alpha+1}$. Since $x$ is the root of $T$, it will only be missing from $T^{\alpha+1}$ if $T^{\alpha+1}$ is empty.
	
	Suppose $T$ is not $\alpha$-fanlike. Let $x$ be the root of $T$. Then $C_T(x) = T$ is not $\alpha$-fanlike, so $x\in T^{\alpha+1}$, and thus $T^{\alpha+1}$ is not empty.
\end{proof}

This pruning process always stabilizes in fewer than $|T|^+$ steps.

\begin{lemma}\label{lem:pruningstabilizes}
	There is an ordinal $\beta < |T|^+$ such that for all $\beta' > \beta$, $T^{\beta'} = T^\beta$.
\end{lemma}
\begin{proof}
	First, note that if $T^\beta = T^{\beta+1}$, then $T^\beta$ has no elements with bounded cones, so $T^{\beta'} = T^\beta$ for all $\beta' >\beta$. That is, if the sequence is constant for even two consecutive elements, it must stabilize completely.
	
	Suppose the pruning does not stabilize at stage $\alpha$, so $T_{\alpha+1}\neq T_\alpha$. Then $T_{\alpha+1}$ is a proper subset of $T_\alpha$, and thus we have removed at least one node. If the process never stabilizes before stage $|T|^+$, then we will have removed $|T|^+$-many nodes from $T$, a contradiction.
\end{proof}

The pruning process will stabilize in one of two ways: either the tree is completely destroyed, or we hit a well-pruned subtree. This gives us the following fact:

\begin{lemma}\label{lem:FLiffWPsubtree}
	A tree $T$ has a well-pruned subtree if and only if it is not fanlike.
\end{lemma}
\begin{proof}
	 Suppose $T$ has a well pruned subtree $T'$. Then $T' \subseteq T^\alpha$ for all $\alpha$: no nodes of $T'$ can be removed at successor stages, because all nodes in $T'$ have unbounded cones in $T'$. At limits we take intersections, so all of $T'$ will also survive limit stages. We conclude that $T^{\alpha+1}$ is nonempty for all $\alpha$, so by Corollary \ref{cor:flpruning}, $T$ cannot be $\alpha$-fanlike for any $\alpha$.
	 
	 Now suppose $T$ is not fanlike. Then it is not $\alpha$-fanlike for any $\alpha$, so by Corollary \ref{cor:flpruning}, $T^\alpha$ is nonempty for all $\alpha$. By Lemma \ref{lem:pruningstabilizes}, the $T^\alpha$'s stabilize at some stage $\beta < |T|^+$. Define $T' = T^\beta$. Since $T^\beta = T^{\beta+1}$ and $T^\beta$ is nonempty, we conclude that all nodes in $T^\beta$ have unbounded cones, and thus $T^\beta$ is a well-pruned subtree of $T$.
\end{proof}

Armed with this perspective on fan degree, we turn now to our main theorem.

\begin{theorem}\label{thm:1cardmain}
	Let $\ka$ be regular. Suppose $T$ is a tree of height $\ka$ that admits a narrow system. Then $T$ contains a well-pruned subtree.
\end{theorem}
\begin{proof}
	Suppose $T$ does not have a well-pruned subtree, but does have a narrow subsystem $S = \langle S_\alpha \mid \alpha \in I \rangle$. Let $\theta = \width(S)$, noting that $\theta^+ < \ka$.
	
	We will prove something slightly stronger than the statement of the theorem. Given $\alpha \in I$, we say that a subset $C\subseteq T$ \emph{hits $\alpha$} if $C\cap S_\alpha$ is nonempty. We claim that there is some $\alpha < \ka$ such that for any selection $A$ of $\leq \theta$-many nodes in $T_\alpha$, the cone above $A$, written $C(A) = \bigcup_{x\in A}C(x)$, misses unboundedly many elements of $I$. In particular, this means that the coherence requirement of the system $S$ must fail: if we set $A = S_\alpha$, then there will be some $\gamma > \alpha$ such that nothing in $S_\gamma$ is above anything in $S_\alpha$.
	
	Suppose that the claim fails for $T$. This means that for all $\alpha < \ka$, there is a set $A_\alpha$  of $\leq\theta$-many nodes in $T_\alpha$ and an ordinal $\gamma_\alpha$ such that $A_\alpha$ hits all levels $\geq \gamma_\alpha$.
	
	Since $T$ has no well-pruned subtree, by Lemma \ref{lem:FLiffWPsubtree} it is $\rho$-fanlike for some ordinal $\rho$.	This means that there is some $\delta$ such that the cones above every node in $T_\delta$ are $\rho'$-fanlike for some $\rho'<\rho$. Set $\alpha_0 = \delta$.
	
	For each element $x\in A_{\alpha_0}$, since $C(x)$ is fanlike, there is some $\delta_x$ such that for all $y \in C(x)$ with $\Lev(y) \geq \delta_x$, $\fd(C(y)) < \fd(C(x))$. Let $\alpha_1 = \sup_{x\in A_{\alpha_0}}\delta_x$.
	
	We repeat this process inductively, taking supremums at limit stages, defining an increasing sequence $\langle \alpha_\beta \mid \beta < \theta^+\rangle$ such that for all $x \in A_{\alpha_\beta}$ and all $y \in C_x$ at or above level $\alpha_{\beta+1}$, $\fd(C(y))<\fd(C(x))$.
	
	\begin{claim}\label{cl:wellfounded}
		For any $x$ in any $C(A_{\alpha_\beta})$, $B_x = \{\beta' \mid x\in C(A_{\alpha_{\beta'}})\}$ is finite.
	\end{claim}
	\begin{proof}
		Suppose $\langle \beta_n \mid n<\w\rangle$ is an increasing sequence of elements in $\theta^+$, and suppose $x \in \bigcap_{n<\w}C(A_{\alpha_{\beta_n}})$. Then there is a  sequence $x_n$ such that each $x_n \in A_{\alpha_{\beta_n}}$, and $x \in C(x_n)$. Since $T$ is a tree, all nodes in this this sequence must be comparable; since their levels are increasing, the sequence is $\leq_T$-increasing.
		
		For all $n<m<\w$, since $x_n \leq_T x_m$ and $x_m$ is above level $\alpha_{\beta_n}$, by construction $\fd(C(x_m)) < \fd(C(x_n))$.
		This gives an infinite decreasing sequence of ordinals, a contradiction.
	\end{proof}
	
	Let $\gamma = \sup_{\beta < \theta^+}\gamma_\beta$. Then for all $\beta < \theta^+$, $C(A_{\alpha_\beta})\cap S_\gamma$ must be nonempty.
	For all $y \in S_\gamma$, consider the set $B_y = \{\beta <\theta^+ \mid y\in C(A_{\alpha_\beta})\}$. Clearly $\bigcup_{y\in S_\gamma} B_y = \theta^+$, since each $C(A_{\alpha_\beta})$ intersects $S_\gamma$. Since $S_\gamma$ has size at most $\theta$, there is some $y \in S_\gamma$ such that $B_y$ has size $\theta^+$.
	This means that $y$ is in $C(A_{\alpha_\beta})$ for infinitely many $\beta$, contradicting Claim \ref{cl:wellfounded}.
\end{proof}

	To put this theorem into context, we first note several cases where it is trivial. If $T$ has a cofinal branch, that branch is itself both a well-pruned subtree of $T$ and a narrow subsystem. Moreover, it is a well-known fact that any thin tree of height $\ka$ has a well-pruned subtree.
	
	The trees for which this theorem is nontrivial are trees of height $\ka$ with levels of size $\geq \ka$ and no cofinal branch. These are often called \emph{wide Aronszajn trees}, although that terminology sometimes refers to trees with width exactly $\ka$. These trees are of interest in several areas of logic; in particular, they are used as game clocks for Ehrenfeucht–Fra\"iss\'e games of uncountable length. For more about wide Aronszajn trees, see \cite{dzamonashelah:wideAtrees} and \cite{kivimaki:universalwideAtrees}.

\section{Some Well-Pruned Wide Aronszajn Trees}
Unlike normal $\ka$-Aronszajn trees, wide $\ka$-Aronszajn trees always exist for any $\ka$. We are interested in the behavior of these trees even when the tree property holds, so we are particularly interested in wide Aronszajn trees that have no thin subtree of height $\ka$. We record two explicit examples of such trees.

\begin{exmp}[The Infinite Fan]
	Recall that when we defined the $\w$-fan, our procedure at limits was aimed at preserving the properties of the $n$-fans. But we could have iterated the construction in a different way. That is, if we view each $n+1$-fan $F_{n+1}$ as extending a $n$-fan $F_n$ by adding a $1$-fan to the end of each branch, we can define a tree $F = \bigcup_{n<\w} F_n$. (In fact, we can continue this process transfinitely, as long as we stop at a limit stage below stage $\ka$; for simplicity, we will stop at stage $\w$.) This should be thought of as iterating the process of adding a $1$-fan to the end of each branch $\w$-many times.

	The tree $F$ has some interesting properties. First, unlike the fanlike trees described in the previous section, it is well-pruned. It is also highly homogeneous: for any node $x$, $C(x)$ is isomorphic to the full tree. It has no cofinal branches, since each branch consists of concatenating countably many segments each of which has length $<\ka$. Moreover this remains true in any extension in which $\ka$ remains a regular cardinal. Finally, we note that $F$ contains no thin subtree of height $\ka$.
\end{exmp}

\begin{exmp}[Specker's Tree]
	A well-known theorem due to Specker \cite{specker}, generalizing the standard construction of an special Aronszajn tree, states that if $\ka^{<\ka} = \ka$, then there is a $\ka^+$-Aronszajn tree. A proof of his theorem using walks on ordinals appears in \cite{kanamori}. The cardinal arithmetic assumption is only used to show that the levels are thin; if it fails, the tree can still be constructed, and is a wide $\ka^+$-Aronszajn tree that is well-pruned. In particular, if the tree property holds at $\ka^+$, then the tree built in this manner will be wide.
\end{exmp}

These trees are branchless well-pruned trees of height $\ka$ that are definable in any model of ZFC. Moreover, their branchlessness is upwards absolute to models that don't collapse or singularize $\ka$. In particular, if the Narrow System Property holds at $\ka$, then these trees cannot admit a narrow system: such a system would necessarily have a cofinal branch, which would induce a cofinal branch in the tree.

It follows that ZFC cannot prove that well-pruned wide Aronszajn trees on any cardinal admit a narrow system, since in models where the Narrow System Property holds, these trees are counterexamples. On the other hand, it is unknown to the author whether it is consistent for either of these trees to admit a narrow system.

\section{Two-Cardinal Trees and Systems}

In this section, we generalize the definitions and results of Section \ref{s:1card} to two-cardinal trees.
\begin{defn}
	A $(\ka,\la)$-tree $T$ is \emph{well-pruned} if for every element $x \in T_z$ and all $z' \supseteq z$, there is $y\in T_{z'}$ such that $y\cap z = x$. That is, every node in $T$ has an extension to all larger levels.
\end{defn}
\begin{defn}
	Let $T$ be a $(\ka,\la)$-tree, and let $C$ be a cone in $T$. We define the notion of \emph{$\alpha$-fanlike} inductively:
	\begin{itemize}
	\item A cone $C$ in $T$ is $0$-fanlike if it is bounded.
	
	\item A cone $C$ above a node $x \in T_{z_0}$ is $\alpha+1$-fanlike if there is $z\in \mc{P}_\ka(\la)$ with $z\supseteq z_0$ such that for each $y\in C\cap T_z$, the cone above $y$ is either bounded in $\mc{P}_\ka(\la)$ or is $\alpha$-fanlike.
	
	\item If $\gamma$ is a limit ordinal, a cone $C$ above a node $x \in T_{z_0}$ is $\gamma$-fanlike if there is $z \in \mc{P}_\ka(\la)$ with $z \supseteq z_0$ such that for all $y \in C\cap T_z$, the cone above $y$ is $\alpha$-fanlike for some $\alpha < \gamma$.
\end{itemize}
	
	We say that a $(\ka,\la)$-tree $T$ is $\alpha$-fanlike if the above definition holds, replacing $C$ with $T$.
\end{defn}

\begin{defn}
	Let $\ka$ be regular, and let $T$ be a $(\ka,\la)$-tree. We say that $T$ is \emph{fanlike} if $T$ is $\alpha$-fanlike for some $\alpha$. The \emph{fan degree} of $T$ is the least such $\alpha$.
\end{defn}

\begin{lemma}\label{lem:2cardfaniffwp}
	Let $\ka$ be regular, and let $T$ be a $(\ka,\la)$-tree. Then $T$ has a well-pruned $(\ka,\la)$-subtree if and only if $T$ is not fanlike.
\end{lemma}
\begin{proof}
	Identical to the proof of Lemma \ref{lem:FLiffWPsubtree}, noting that all lemmas used in that proof also generalize to $(\ka,\la)$-trees with identical proofs.
\end{proof}

\begin{thm}
	Let $\ka$ be regular. Suppose $T$ is a $(\ka,\la)$-tree that admits a narrow system. Then $T$ contains a well-pruned $(\ka,\la)$-subtree.
\end{thm}
\begin{proof}
	Suppose $T$ does not have a well-pruned subtree, but contains a narrow subsystem $S = \langle S_z \mid z\in I$. Let $\theta = \width(S)$, with $\theta^+<\ka$. 
	
	For $z\in I$, we say that a subset $C\subseteq T$ \emph{hits $z$} if $C\cap S_z$ is nonempty, and \emph{misses $z$} if the intersection is empty. As in the proof of Theorem \ref{thm:1cardmain}, we will prove that there is some $z\in I$ such that for any selection $A$ of $\leq \theta$-many nodes in $T_z$, the union of the cones above elements of $A$ misses $\subseteq$-unboundedly many elements of $I$. As before, this will imply that $S$ cannot be a system.
		
	Suppose that this is not the case. Then for all $z\in I$, there is a set $A_z$ of $\leq \theta$-many nodes in $T_z$ and a set $g_z \in \mc{P}_\ka(\la)$ such that $A_z$ hits $z'$ for all $z' \supseteq g_z$.
	
	Since $T$ has no well-pruned subtree, by Lemma \ref{lem:2cardfaniffwp} it is fanlike. Then there is $b\in \mc{P}_\ka(\la)$ such that for all $x \in T_b$, $\fd(C(x)) < \fd(T)$. Set $a_0 = b$.
	
	For each element $x\in A_{a_0}$, since $C(x)$ is fanlike, there is $b_x$ such that for all $y \in C(x)\cap T_{b_x}$, $\fd(y) < \fd(x)$. Let $a_1 = \bigcup_{x\in A_{\alpha_0}} b_x$. Note that $A_{\alpha_0}$ has size $\leq \theta$, so $a_1 \in \mc{P}_\ka(\la)$.
		
	We repeat this process $\theta^+$-many times, taking unions at limit stages, to define a $\subseteq$-increasing sequence $\langle a_\beta \mid \beta < \theta^+\rangle$ such that for all $x\in A_{a_\beta}$, $z\supseteq a_{\beta+1}$, and all $y\in C(x) \cap T_z$, $\fd(C(y)) < \fd(C(x))$.
	
	\begin{claim}
		For any $x$ in any $C(A_{a_\beta})$, $x$ appears in only finitely many cones $C(A_{a_{\beta'}})$.
	\end{claim}
	\begin{proof}
		Suppose that $\langle \beta_n \mid n<\w\rangle$ is an increasing sequence of ordinals less than $\theta^+$ such that $x \in \bigcap_{n<\w} C(A_{a_{\beta_n}})$. Then for all $n<\w$, there is some $x_n$ such that $x_n \in A_{a_{\beta_n}}$ and $x\in C(x_n)$. The sequence $\langle x_n \mid n<\w\rangle$ must be $\subseteq$-increasing, so in particular given $n<m$, $x_m \in C(x_m)$, above level $\beta_n$. We conclude that $\fd(x_m) < \fd(x_n)$ whenever $n<m$. This gives an infinite decreasing sequence of ordinals, a contradiction.
	\end{proof}
	
	Now let $g = \bigcup_{\beta < \theta^+} g_{a_\beta}$. Noting that $\theta^+<\ka$, $g \in \mc{P}_\ka(\la)$. Then for all $\beta < \theta^+$, $C(A_{a_\beta})\cap S_{g}$ is nonempty.
	For all $y \in S_{g}$, consider the set $B_y = \{\beta < \theta^+ \mid y\in C(A_{a_\beta})\}$. Since each $C(A_{a_\beta})$ hits $g$, $\bigcup_{\gamma\in S_\gamma}B_y  = \theta^+$. Since $S_{g}$ has size at most $\theta$, there is some $y\in S_\gamma$ such that $B_y$ has size $\theta^+$. This means that $y\in C(A_{a_\beta})$ for infinitely many $\beta$, a contradiction.
	
\end{proof}

\bibliography{bib}

@article {MS_TPSuccSing,
    AUTHOR = {Magidor, Menachem and Shelah, Saharon},
     TITLE = {The tree property at successors of singular cardinals},
   JOURNAL = {Arch. Math. Logic},
  FJOURNAL = {Archive for Mathematical Logic},
    VOLUME = {35},
      YEAR = {1996},
    NUMBER = {5-6},
     PAGES = {385--404},
      ISSN = {0933-5846},
   MRCLASS = {03E05 (03E35 04A20)},
  MRNUMBER = {1420265},
MRREVIEWER = {Marion Scheepers},
       DOI = {10.1007/s001530050052},
       URL = {https://doi-org.proxy.cc.uic.edu/10.1007/s001530050052},
}

@article {jech:combinatorialprobs,
    AUTHOR = {Jech, Thomas J.},
     TITLE = {Some combinatorial problems concerning uncountable cardinals},
   JOURNAL = {Ann. Math. Logic},
  FJOURNAL = {Annals of Mathematical Logic},
    VOLUME = {5},
      YEAR = {1972/73},
     PAGES = {165--198},
      ISSN = {0003-4843},
   MRCLASS = {04A20 (02K35)},
  MRNUMBER = {325397},
MRREVIEWER = {Paul E. Howard},
       DOI = {10.1016/0003-4843(73)90014-4},
       URL = {https://doi-org.proxy.cc.uic.edu/10.1016/0003-4843(73)90014-4},
}

@article {specker,
    AUTHOR = {Specker, E.},
     TITLE = {Sur un probl\`eme de {S}ikorski},
   JOURNAL = {Colloq. Math.},
  FJOURNAL = {Colloquium Mathematicum},
    VOLUME = {2},
      YEAR = {1949},
     PAGES = {9--12},
      ISSN = {0010-1354,1730-6302},
   MRCLASS = {27.2X},
  MRNUMBER = {39779},
MRREVIEWER = {F.\ Bagemihl},
       DOI = {10.4064/cm-2-1-9-12},
       URL = {https://doi.org/10.4064/cm-2-1-9-12},
}

@article {lambie-hanson:narrowsystems,
    AUTHOR = {Lambie-Hanson, Chris},
     TITLE = {Narrow systems revisited},
   JOURNAL = {Bull. Lond. Math. Soc.},
  FJOURNAL = {Bulletin of the London Mathematical Society},
    VOLUME = {56},
      YEAR = {2024},
    NUMBER = {6},
     PAGES = {1967--1987},
      ISSN = {0024-6093,1469-2120},
   MRCLASS = {03E55 (03E04 03E05 03E35)},
  MRNUMBER = {4761131},
MRREVIEWER = {Radek\ Honz\'ik},
       DOI = {10.1112/blms.13037},
       URL = {https://doi.org/10.1112/blms.13037},
}

@inproceedings {solovay:sch,
    AUTHOR = {Solovay, Robert M.},
     TITLE = {Strongly compact cardinals and the {GCH}},
 BOOKTITLE = {Proceedings of the {T}arski {S}ymposium ({P}roc. {S}ympos.
              {P}ure {M}ath., {V}ol. {XXV}, {U}niv. {C}alifornia,
              {B}erkeley, {C}alif., 1971)},
    SERIES = {Proc. Sympos. Pure Math.},
    VOLUME = {Vol. XXV},
     PAGES = {365--372},
 PUBLISHER = {Published for the Association for Symbolic Logic by the American Mathematical Society, Providence, RI},
      YEAR = {1974},
   MRCLASS = {02K35 (02H13 02K25)},
  MRNUMBER = {379200},
MRREVIEWER = {F.\ R.\ Drake},
}

@misc{kivimaki:universalwideAtrees,
	title={Universal Wide Aronszajn tree}, 
	author={Siiri Kivimäki},
	year={2025},
	eprint={2511.06526},
	archivePrefix={arXiv},
	primaryClass={math.LO},
	url={https://arxiv.org/abs/2511.06526}, 
}

@article{dzamonashelah:wideAtrees,
	 title={On Wide Aronszajn Trees in the Presence of MA}, volume={86},
	 DOI={10.1017/jsl.2020.42}, 
	 number={1}, 
	 journal={The Journal of Symbolic Logic}, author={D\v{z}amonja, Mirna and Shelah, Saharon}, year={2021}, 
	 pages={210–223}
	 }

@book {kanamori,
	AUTHOR = {Kanamori, Akihiro},
	TITLE = {The higher infinite},
	SERIES = {Springer Monographs in Mathematics},
	EDITION = {Second},
	NOTE = {Large cardinals in set theory from their beginnings,
	Paperback reprint of the 2003 edition},
	PUBLISHER = {Springer-Verlag, Berlin},
	YEAR = {2009},
	PAGES = {xxii+536},
	ISBN = {978-3-540-88866-6},
	MRCLASS = {03E55 (01A55 01A60 03-02 03-03)},
	MRNUMBER = {2731169},
}

@article {lambie-hanson:squares,
	AUTHOR = {Lambie-Hanson, Chris},
	TITLE = {Squares and narrow systems},
	JOURNAL = {J. Symb. Log.},
	FJOURNAL = {The Journal of Symbolic Logic},
	VOLUME = {82},
	YEAR = {2017},
	NUMBER = {3},
	PAGES = {834--859},
	ISSN = {0022-4812,1943-5886},
	MRCLASS = {03E35 (03E05 03E55)},
	MRNUMBER = {3694332},
	MRREVIEWER = {Yair\ Hayut},
	DOI = {10.1017/jsl.2017.38},
	URL = {https://doi.org/10.1017/jsl.2017.38},
}

@article {viale:PFA,
	AUTHOR = {Viale, Matteo},
	TITLE = {The proper forcing axiom and the singular cardinal hypothesis},
	JOURNAL = {J. Symbolic Logic},
	FJOURNAL = {The Journal of Symbolic Logic},
	VOLUME = {71},
	YEAR = {2006},
	NUMBER = {2},
	PAGES = {473--479},
	ISSN = {0022-4812,1943-5886},
	MRCLASS = {03E05 (03E10 03E65 03E75)},
	MRNUMBER = {2225888},
	MRREVIEWER = {Justin\ Tatch\ Moore},
	DOI = {10.2178/jsl/1146620153},
	URL = {https://doi.org/10.2178/jsl/1146620153},
}

@article {viale:guessing,
	AUTHOR = {Viale, Matteo},
	TITLE = {Guessing models and generalized {L}aver diamond},
	JOURNAL = {Ann. Pure Appl. Logic},
	FJOURNAL = {Annals of Pure and Applied Logic},
	VOLUME = {163},
	YEAR = {2012},
	NUMBER = {11},
	PAGES = {1660--1678},
	ISSN = {0168-0072,1873-2461},
	MRCLASS = {03E55 (03E05 03E35 03E65)},
	MRNUMBER = {2959666},
	MRREVIEWER = {A.\ Kanamori},
	DOI = {10.1016/j.apal.2011.12.015},
	URL = {https://doi.org/10.1016/j.apal.2011.12.015},
}

@article {krueger:guessing,
	AUTHOR = {Krueger, John},
	TITLE = {Guessing models imply the singular cardinal hypothesis},
	JOURNAL = {Proc. Amer. Math. Soc.},
	FJOURNAL = {Proceedings of the American Mathematical Society},
	VOLUME = {147},
	YEAR = {2019},
	NUMBER = {12},
	PAGES = {5427--5434},
	ISSN = {0002-9939,1088-6826},
	MRCLASS = {03E05 (03E40)},
	MRNUMBER = {4021101},
	MRREVIEWER = {Peter\ Holy},
	DOI = {10.1090/proc/14739},
	URL = {https://doi.org/10.1090/proc/14739},
}
\bibliographystyle{plain}

\end{document}